\documentclass[10pt]{article}
\usepackage{mathrsfs}
\usepackage{amsmath,amscd,amsthm,amsfonts}
\usepackage{graphicx}
\oddsidemargin=0.5cm\topmargin=-1cm \numberwithin{equation}{section}

\usepackage{amsmath}
  \usepackage{paralist}
  \usepackage{graphics} 
  \usepackage{epsfig} 
\usepackage{graphicx}  \usepackage{epstopdf}
 \usepackage[colorlinks=true]{hyperref}
\hypersetup{urlcolor=blue, citecolor=red}

\usepackage{bbm,,color,enumitem,mathrsfs} 
\newtheorem{theorem}{Theorem}[section]

\newtheorem{lemma}[theorem]{Lemma}
\newtheorem{proposition}{Proposition}

\theoremstyle{definition}
\newtheorem{definition}[theorem]{Definition}
\newtheorem{remark}{Remark}

\newtheorem{assumption}{Assumption} 
\newtheorem{theoremm}{Theorem}  
\newtheorem*{fact}{Fact} 

\begin{document}
\vskip 5.1cm
\title{Formula of Entropy along Unstable Foliations for $C^1$ Diffeomorphisms with Dominated Splitting
\footnotetext{
It is supported by NSFC(No:11371120, 11771118).}}
\author {Xinsheng Wang$^1$, Lin Wang$^2$ and Yujun Zhu$^3$ \\
\small {1. College of Mathematics and Information Science}\\
\small {Hebei Normal University, Shijiazhuang, 050024, China}\\
\small {2. School of Applied Mathematics}\\
\small {Shanxi University of Finance and Economics, Taiyuan, 030031, China}\\
\small {3. School of Mathematical Sciences}\\
\small {Xiamen University, Xiamen, 361005, China}
}
\date{}
\maketitle

\begin{abstract}
Metric entropies along a hierarchy of unstable foliations are investigated for $C^1$ diffeomorphisms with dominated splitting. The analogues of Ruelle's inequality and Pesin's formula, which relate the metric entropy and Lyapunov exponents in each hierarchy, are given.
\end{abstract}

\section{Introduction}
It is well known that among the major concepts of smooth ergodic theory are the notions of invariant measures, entropy and Lyapunov exponents. Entropies, including measure-theoretic entropy and topological entropy, play important roles in the study of the complexity of a dynamical system. Intuitively, topological entropy measures the exponential growth rate in $n$ of the number of orbits of length $n$ up to a small error, measure-theoretic entropy gives the maximum average information with respect to some invariant measure one can get from a system. While Lyapunov exponents reflect the rate at which two nearby orbits separate from each other. What interests one is the relation between entropy and Lyapunov exponents. Let $f$ be a $C^1$ diffeomorphism on a compact Riemannian manifold $M$ without boundary. For any regular point in the sense of Oseledec \cite{Oseledec1968} $x\in M$, let $\lambda_1(x)>\lambda_2(x)>\dots>\lambda_{r(x)}(x)$ denote its distinct Lyapunov exponents, and $E_1(x)\oplus\dots\oplus E_{r(x)}(x)$ be the corresponding decomposition of its tangent space $T_xM$. In 1970s, Ruelle \cite{Ruelle1978} gave the following inequality
\[
h_{\mu}(f)\leq\int_{M}\sum_{\lambda_i(x)>0}\lambda_i(x)m_i(x)\mathrm{d}\mu(x)
\]
for any $f$-invariant measure $\mu$, where $m_i(x)=\mathrm{dim}E_i(x)$. Moreover, if $f$ is $C^2$ and $\mu$ is equivalent to the Riemannian measure on $M$, Pesin \cite{Pesin1977} proved that the equality (which is called Pesin's entropy formula) holds in the above inequality.

In 1981 Ma\~n\'e \cite{Mane1981} gave an ingenious approach to prove Pesin's entropy formula under the assumption that $f$ is $C^{1+\alpha}$($\alpha>0$) and $\mu$ is absolutely continuous with respect to the Lebesgue measure. In 1985, Ledrappier and Young \cite{LedrappierYoung1985} proved that Pesin's entropy formula holds for $C^2$ diffeomorphisms if and only if $\mu$ is an SRB measure. Furthermore, in \cite{LedrappierYoung1985a} they gave a more general formula which is called the dimension formula for any $f$-invariant measure $\mu$ as follows
\[
h_{\mu}(f)=\int_{M}\sum_{\lambda_i(x)>0}\lambda_i(x)\gamma_i(x)\mathrm{d}\mu(x),
\]
where $\gamma_i(x)$ denotes the dimension of $\mu$ in the direction of the subspace $E_i(x)$. In their argument, they used the notion of `` entropy along unstable foliations '', which reflects the complexity of the system at different levels.

Except for Ruelle's inequality, all other results above require that $f$ is $C^{1+\alpha}$ or $C^2$, so it is interesting to investigate Pesin's formula under $C^1$ differentiability hypothesis plus some additional conditions, for example, dominated splitting. Recently, Sun and Tian \cite{SunTian2010} applied Ma\~n\'e's method to prove that Pesin's entropy formula holds if $f$ is a $C^1$ diffeomorphism with dominated splitting. In \cite{CatsigerasCerminaraEnrich2015}, Catsigeras, Cerminara, and Enrich considered a nonempty set of invariant measures which describe the asymptotic statistics of Lebesgue almost all orbits, and they proved that the measure-theoretic entropy of each of these measures is bounded from below by the sum of the Lyapunov exponents on the dominating subbundle. For more details about the dynamics of a system with dominated splitting, one can refer to  \cite{Pujals2007} and  \cite{Sambarino2014}. Instead of the condition in \cite{SunTian2010} that $f$ admits a dominated splitting, Tian \cite{Tian2014} gave the concept of nonuniformly-H\"older-continuity for an $f$-invariant measure, and proved that Pesin's entropy formula holds under that assumption.

An interesting question is: can we get the formula of entropy along unstable foliations for a $C^1$ diffeomorphism with dominated splitting? In this paper, we give a positive answer of this question. The analogues of Ruelle inequality and Pesin's entropy formula are given. In the proofs, we borrow some idea from Hu, Hua and Wu's paper \cite{Wu2015} in which a variational principle relating the topological entropy and measure-theoretic entropy on the unstable foliation of a partially hyperbolic diffeomorphism is obtained.

This paper is organized as follows. In section \ref{MainResults}, we give some preliminaries and the statement of our main results, and the proofs of the main results are given in the next two sections.

\section{Preliminaries and Statement of Results}\label{MainResults}
Throughout this paper, Let $M$ be a compact Riemannian manifold without boundary, $f$ a $C^1$ diffeomorphism on $M$, and $\mu$ an $f$-invariant Borel measure.

Let $\Gamma$ be the set of points which are regular in the sense of Oseledec \cite{Oseledec1968}. For $x\in\Gamma$, let
\[
\lambda_1(x)>\lambda_2(x)>\dots>\lambda_{r(x)}(x)
\]
denote its distinct Lyapunov exponents and let
\[
T_xM=E_1(x)\oplus\dots\oplus E_{r(x)}(x)
\]
be the corresponding decomposition of its tangent space.

Now we give the definition of dominated splitting. Denote the minimal norm of an invertible linear map $A$ by $m(A)=\Vert A^{-1}\Vert^{-1}$.

\begin{definition}
\begin{enumerate}[fullwidth,label=(\arabic*)]
\item(Dominated splitting at one point) Let $x \in M$ and $T_{x}M=E(x) \oplus F(x)$ be a $Df$-invariant splitting on $orb(x)$. $T_{x}M=E(x) \oplus F(x)$ is called to \emph{be $(N(x),i(x))$-dominated splitting at $x$}, if the dimension of $F$ is $i(x)$($1\leq i(x)\leq\dim M-1$) and there exists a constant $N(x) \in{\mathbbm Z}^+$ such that
\[
  \frac{\Vert Df^{N(x)}|_{E(f^j(x))}\Vert}{m(Df^{N(x)}|_{F(f^j(x))})}\leq\frac{1}{2},\forall j\in\mathbbm Z.
\]

\itemindent=1.25em\item(Dominated splitting on an invariant set) Let $\Delta$ be an $f$-invariant set and $T_{\Delta}M=E\oplus F$ be a $Df$-invariant splitting on $\Delta$. We call $T_{\Delta}M=E\oplus F$ to \emph{be $(N,i(y))$-dominated splitting}, if the dimension of F at $y$ is $i(y)$($1\leq i(y)\leq\dim M-1$) and there exists a constant $N\in\mathbbm Z^+$ such that
\[
  \frac{\Vert Df^{N}|_{E(y)}\Vert}{m(Df^{N}|_{F(y)})}\leq\frac{1}{2},\forall y\in\Delta.
\]
\end{enumerate}
\end{definition}

In the following, we consider two cases of the invariant measure $\mu$.
\begin{enumerate}[fullwidth,listparindent=1.25em,label=\textbf{Case \arabic*}]
\item $\mu$ is ergodic. In this case, the functions $x\mapsto r(x)$, $\lambda_i(x)$ and $\dim E_i(x)$ are constant $\mu$-a.e., denote them by $r$, $\lambda_i$ and $m_i$ respectively. Let $u=\max\{i\colon\lambda_i>0\}$, $u(i)=u-i+1$, and $\delta_*=\min\{\lambda_{i+1}-\lambda_i,1\leq i\leq r-1\}$.
For $1\leq i\leq u$, $x \in\Gamma$ and $0<\varepsilon<\delta_*$, we define
\begin{equation}\label{manifold}
W^i(x)=\{y \in M\colon\underset{n\to\infty}{\limsup}\frac{1}{n}\log d(f^{-n}x,f^{-n}y)\leq-\lambda_{u(i)}+\varepsilon\},
\end{equation}
where $d$ is the Riemannian metric on $M$. The following result ensures that $W^i(x)$ is an immersed $C^1$-manifold under the assumption of dominated splitting.

\begin{proposition}[{\cite[Proposition 8.9]{AbdenurBonattiCrovisier2011}}]\label{stablemanifold}
Let $\mu$ be an ergodic measure whose support admits a dominated splitting $E\oplus F$, let $\lambda_E^+<\lambda_F^-$ be the maximal Lyapunov exponent in $E$ and the minimal Lyapunov exponent in $F$ of the measure $\mu$.

If $\lambda_E^+$ is strictly negative, then at $\mu$-a.e. point $x\in M$, there exists an injectively immersed $C^1$-manifold $W^E(x)$ with $\dim W^E(x)=\dim E$, tangent to $E_x$, which is a stable manifold, and for any $\lambda\leq0$ contained in $(\lambda_E^+,\lambda_F^-)$ and $\mu$-a.e. point $x$ we have
\[
W^E(x)=\{y\in M\colon d(f^nx,f^ny)e^{-\lambda n}\to0, \text{as }n\to\infty\}.
\]
\end{proposition}
For $x\in\Gamma$ and $1\leq i\leq u$, let
\[
E(x)=\underset{u(i)+1\leq j\leq r}{\oplus} E_j(x) \quad\text{and}\quad F(x)=\underset{1\leq j\leq u(i)}{\oplus}E_j(x).
\]
\begin{assumption}[\textbf{ergodic case}]\label{ergodic_case}
  For each $1\leq i\leq u$, $T_{\Gamma}M=E\oplus F$ is $(N,I(i))$-dominated splitting for some $N\in \mathbb{Z}^+$, where $I(i)=\sum_{1\leq j\leq u(i)}m_j$.
\end{assumption}

\begin{remark}
It is obvious that any hyperbolic automorphism on two dimensional tori satisfies Assumption \ref{ergodic_case}.
\end{remark}

Under Assumption \ref{ergodic_case}, we know that, by Proposition \ref{stablemanifold}, $W^i(x)$ is a $C^1$ $I(i)$-dimensional immersed submanifold of $M$ tangent at $x$ to $F(x)$ by replacing $f$ with $f^{-1}$. It is called \emph{the $i$th unstable manifold} of $f$ at $x$. $\{W^i(x)\colon x\in\Gamma\}$ is called \emph{the $W^i$-foliation}.

A measurable partition $\xi^i$ of $M$ is said to \emph{be subordinate to} $W^i$ if for $\mu$ -a.e. $x \in \Gamma$, $\xi^i(x)\subset W^i(x)$ and contains an open neighborhood of $x$ in $W^i(x)$. An important property with respect to such a partition is that there is a canonical system of conditional measures $\{\mu_x^{\xi^i}\}$. The following lemma ensures the existence of such partitions.
\begin{lemma}\label{specialpartition}
Let $\mu$ be an ergodic measure, then there exists a measurable partition $\xi^i$ of $M$ satisfying the following properties:
\begin{enumerate}[label=(\arabic*)]
  \item $\xi^i$ is a partition subordinate to $W^i$;
  \item $\xi^i$ is increasing, i.e., $f^{-1}\xi^i\geq\xi^i$;
  \item $\overset{\infty}{\underset{n=1}{\bigvee}}f^{-n}\xi^i=\epsilon$, where $\epsilon$ is the partition of $M$ into points;
  \item $\overset{\infty}{\underset{n=1}{\bigwedge}}f^{n}\xi^i=\mathscr H(\Pi^+)$, where $\Pi^+$ is the partition of $M$ into global $i$th unstable manifolds, and $\mathscr H(\Pi^+)$ is the measurable hull of $\Pi^+$.
\end{enumerate}
\end{lemma}
\begin{proof}
For the proof, the reader can refer to \cite{PesinSinai1982}.
\end{proof}

For more details about measurable partitions and conditional measures the reader can refer to Section 0.1 -- 0.3 in \cite{LiuQian1995} and Section 3 and 4 in \cite{Rokhlin1962}.

Let $\xi^i$ be a measurable partition subordinate to $W^i$ with conditional measures $\{\mu_x^{\xi^i}\}$. Define $h^i_{\mu}(f,x)\colon \Gamma \to \mathbbm R$ by
\begin{align*}
  h^i_{\mu}(f,x) &=h^i_{\mu}(f,x,\xi^i) \\
  {} &= h_i(f,x,\xi^i,\{\mu_x^{\xi^i}\}) \\
  {} &= \lim_{\varepsilon\to 0}\liminf_{n\to\infty}-\frac{1}{n}\log\mu_x^{\xi^i}V^i(f,x,n,\varepsilon) \\
  {} &= \lim_{\varepsilon\to 0}\limsup_{n\to\infty}-\frac{1}{n}\log\mu_x^{\xi^i}V^i(f,x,n,\varepsilon),
\end{align*}
where $V^i(f,x,n,\varepsilon)\colon =\{y \in W^i(x)\colon d^i(f^k(y),f^k(x))<\varepsilon,0\leq k\leq n-1\}$, and $d^i$ is the metric on $W^i(x)$ given by the Riemannian structure inherited from $M$.

$h_{\mu}^i(f,x)$ is well defined and is independent of the choice of $\xi^i$ or $\mu^{\xi^i}_x$, and it is easy to verify that $h_{\mu}^i(f,x)=h_{\mu}^i(f,fx)$ $\mu$-a.e.(cf.\cite{LedrappierYoung1985a}). Since $\mu$ is ergodic, $h_{\mu}^i(f,x)$ is constant $\mu$-a.e.
\begin{definition}\label{entropydef}
We define \emph{the entropy of $f$ along $i$th unstable foliation} by
\[
h_{\mu}^i(f)=\int_{M}h^i_{\mu}(f,x)\mathrm{d}\mu(x).
\]
\label{entropy}

Since $\mu$ is ergodic, we know that $h_{\mu}^i(f,x)=h_{\mu}^i(f)$, for $\mu$-a.e.$x\in M$.
\end{definition}

\item $\mu$ is arbitrary. In this case, the functions $x\mapsto r(x)$, $\lambda_i(x)$ and $\dim E_i(x)$ are now measurable. Let $u(x)=\max\{i\colon\lambda_i(x)>0\}$, $u(i,x)=u(x)-i+1$, and $\Gamma_i=\{x\in\Gamma\colon u(i,x)>0\}$. Then we can define $W^{u(i,x)}(x)$ as in (\ref{manifold}) except that $u(i)$ and $\lambda_{u(i)}$ should be replaced by $u(i,x)$ and $\lambda_{u(i,x)}(x)$ respectively, and the choice of $\varepsilon$ depends on $x$ such that $\varepsilon<\lambda_{u(i,x)}(x)-\lambda_{u(i,x)+1}(x)$.

For $x\in\Gamma$ and $1\leq i\leq u(x)$, let
\[
E(x)=\underset{u(i,x)+1\leq j\leq r(x)}{\oplus}E_j(x) \quad\text{and}\quad F(x)=\underset{1\leq j\leq u(i,x)}{\oplus}E_j(x).
\]
\begin{assumption}[\textbf{general case}]\label{general_case}
For each $x\in\Gamma$ and $1\leq i\leq u(x)$, $T_{orb(x)}M=E\oplus F$ is $(N(i,x),I(i,x))$-dominated splitting, where $N(i,\cdot)\colon \Gamma_i\to \mathbb{Z}^+$ is a measurable function and $I(i,x)=\sum_{1\leq j\leq u(i,x)}m_j(x)$.
\end{assumption}

Similar to that in Proposition \ref{stablemanifold}, under Assumption \ref{general_case}, $W^i(x)$ is a $C^1$ $I(i,x)$-dominated immersed submanifold of $M$ tangent to $F(x)$ at $x$. It is also called \emph{the $i$th unstable manifold of $f$ at $x$}. $\{ W^i(x)\colon x\in\Gamma_i\}$ is called \emph{the $W^i$-foliation}.

A measurable partition $\xi^i$ of $M$ is said to \emph{be subordinate to $W^i$ on $\Gamma_i$}, if for $\mu$-a.e. $x \in\Gamma_i$, $\xi^i(x)\subset W^i(x)$ and contains an open neighborhood of $x$ in $W^i(x)$. For the existence of such $\xi^i$, one can simply disintegrate $\mu$ into its ergodic components and note that the entire leaf $W^i(x)$ is contained in the ergodic component of $x$(cf. \cite{LedrappierYoung1985a}). There is a canonical system of conditional measures $\{ \mu^i_x\}$ as $\mu$ is ergodic. Then for $x\in\Gamma_i$, we can define $h^i_{\mu}(f,x)\colon \Gamma_i\to \mathbbm R$ as $\mu$ is ergodic. And we define \emph{the entropy along $W^i$ on $\Gamma_i$} which is still denoted by $h^i_\mu(f)$ as in Definition \ref{entropydef} except that $M$ now should be replaced by $\Gamma_i$.
\end{enumerate}
\begin{remark}
It is easy to check that when $\mu$ is ergodic, $\Gamma_i=\Gamma$ for $1\leq i\leq u$. So when $\mu$ is ergodic, the entropy along $i$th unstable foliation on $\Gamma_i$ coincides with the entropy along $i$th unstable foliation. So we call the entropy defined as above \emph{the entropy of $f$ along $i$th unstable foliation}.
\end{remark}
\textbf{Standing hypotheses for the remaining of this paper:} When $\mu$ is ergodic, we set Assumption \ref{ergodic_case}, and when $\mu$ is arbitrary, we set Assumption \ref{general_case}.

Now we are ready to state our main results of this paper:
\begin{theoremm}\label{Ruelle}
  Let $\mu$ be an invariant measure. Then we have the following inequality
  \[
  h^i_{\mu}(f)\leq\int_{\Gamma_i}\sum_{j\leq u(i,x)}m_j(x)\lambda_j(x)\mathrm{d}\mu(x).
  \]

In particular, if $\mu$ is ergodic, then $h_{\mu}^i(f,x)$, $\lambda_i(x)$, $m_i(x)$ and $u(i,x)$ are constant, then we have
  \[
  h_{\mu}^i(f)\leq\sum_{j\leq u(i)}m_j\lambda_j.
  \]
\end{theoremm}
Moreover, if $\mu$ satisfies some additional conditions, we have the following theorem.
\begin{theoremm}\label{mainresult}
  Let $\mu$ be an invariant measure satisfying that for $\mu$-a.e. $x\in M$ and every measurable partition $\xi^i$ subordinate to $W^i$, $\mu^{\xi^i}_x\ll \lambda^i_x$, where $\lambda_x^i$ is the corresponding Riemannian measure on $W^i(x)$. Then we have the following entropy formula
  \[
  h^i_{\mu}(f)=\int_{\Gamma_i}\sum_{j\leq u(i,x)}m_j(x)\lambda_j(x)\mathrm{d}\mu(x).
  \]

In particular, if $\mu$ is ergodic, then we have
  \[
  h_{\mu}^i(f)=\sum_{j\leq u(i)}m_j\lambda_j.
  \]
\end{theoremm}
\begin{remark}
  We only need to prove the ergodic versions of Theorem \ref{Ruelle} and Theorem \ref{mainresult} respectively, and the nonergodic versions of them follow immediately from the ergodic versions by decomposing $\mu$ into ergodic components (just as that has been done in \cite{LedrappierYoung1985a}). So in the following two sections, we always assume that $\mu$ is ergodic.
\end{remark}

In the following, we relate the entropy $h^i_\mu(f)$ along the unstable foliation $W^i$ with the supremum of certain conditional entropy of finite partitions with respect to a measurable partition subordinate to $W^i$. This idea derives from \cite{Wu2015}.
\begin{definition}\label{generalpartition}
Let $\mu$ be an ergodic $f$-invariant measure and $\mathscr P_M$ denote the set of all finite Borel partitions of $M$. \emph{The conditional entropy of $\alpha \in \mathscr P_M$ with respect to $\xi^i$} is defined as
\[
H^i(\alpha |\xi^i)=\int_{\Gamma_i}-\log\mu^{\xi^i}_x(\alpha(x)\cap\xi^i(x))\mathrm{d}\mu(x).
\]
\end{definition}

The following proposition gives an equivalent definition of $h_{\mu}^i(f)$.
\begin{proposition}\label{equivalence}Let $\mu$ be an ergodic measure, then we have
\[
h_{\mu}^i(f)=\underset{\alpha \in \mathscr P_M}{\sup}\underset{n\to \infty}{\limsup}\frac{1}{n}H^i(\overset{n-1}{\underset{i=0}{\bigvee}}f^{-i}\alpha|\xi^i).
\]
\end{proposition}
\begin{proof}
Similar to the proof of $h_{\mu}^u(f)=\underset{\alpha \in \mathscr P_M}{\sup}\underset{n\to \infty}{\limsup}\frac{1}{n}H^u(\overset{n-1}{\underset{i=0}{\bigvee}}f^{-i}\alpha|\xi)$ in \cite{Wu2015}, where $\xi$ is a partition subordinate to the unstable foliation $W^u$. We omit the details.
\end{proof}

\begin{remark}
  In fact, the partitions used in Definition \ref{generalpartition} and Proposition \ref{equivalence} can be replaced by some more natural partitions. Roughly speaking, such partition is constructed via the intersection of a finite partition and the local unstable manifolds. For more details, the reader can refer to \cite{Wu2015}.
\end{remark}

As the classical measure-theoretic entropy and the topological entropy, the entropy along $i$th unstable foliation also has the so-called power rule.
\begin{proposition}[Power rule]
  For $m\geq1$, we have
  \begin{equation}\label{powerrule}
  h^i_{\mu}(f^m)=mh^i_{\mu}(f).
  \end{equation}
\end{proposition}
\begin{proof}
  Let $\xi$ be a measurable partition of $M$ subordinate to $W^i$. Fix $\varepsilon>0$. It is clear that
  \[
  V^i(f^m,x,n,\varepsilon)\supset V^i(f,x,mn,\varepsilon),
  \]
  so, we have
  \[
  -\frac{1}{n}\log\mu_x^{\xi}V^i(f^m,x,n,\varepsilon)\leq-m\frac{1}{mn}\log \mu_i^{\xi}V^i(f,x,mn,\varepsilon).
  \]
  Let $n\to \infty$, and then $\varepsilon \to 0$, we obtain
  \begin{equation}\label{leq}
  h_{\mu}^i(f^m,x)\leq mh^i_{\mu}(f,x).
  \end{equation}

  On the other hand, pick $\delta_0>0$, define $W^i(x,\delta_0)=\{y\colon y\in W^i(x), d^i(y,x)<\delta_0\}$. Because of the compactness of $\overline{W^i(x,\delta_0)}$, we can pick $0\leq \delta\leq\varepsilon<\delta_0$ such that if $d^i(x,y)\leq \delta$, we have
  \[
  d^i(f^j(x),f^j(y))\leq\varepsilon, \forall 0\leq j \leq m-1.
  \]
  It follows that
  \[
  V^i(f^m,x,n,\delta)\subset V^i(f,x,mn,\varepsilon),
  \]
  and hence,
  \[
  -\frac{1}{n}\log\mu_x^{\xi}V^i(f^m,x,n,\delta)\geq-m\frac{1}{mn}\log \mu_i^{\xi}V^i(f,x,mn,\varepsilon).
  \]
  Let $n\to \infty$, and then $\varepsilon \to 0$ (hence $\delta \to 0$), we obtain
  \begin{equation}\label{geq}
  h_{\mu}^i(f^m,x)\geq mh^i_{\mu}(f,x).
  \end{equation}

  (\ref{powerrule}) follows from (\ref{leq}) and (\ref{geq}) immediately.
\end{proof}
\section{Proof of Theorem \ref{Ruelle}}
Now we complete the proof of Theorem \ref{Ruelle}. Firstly, we need the following definition from  \cite{Wu2015}.
  \begin{definition}
    Pick $0<\delta<r$, where $r$ is as in the proof of Lemma \ref{specialpartition}, and $\varepsilon>0$ small enough. Let $S\subseteq\overline{W^i(x,\delta)}$ satisfying
    \[
    d(f^jy,f^jz)\geq\varepsilon, \exists 0\leq j\leq n-1, \forall y,z\in S,
    \]
    we call $S$ an $(n,\varepsilon)$ $i$-separated set of $\overline{W^i(x,\delta)}$. Let $N^i(f,\varepsilon,n,x,\delta)$ denote the largest cardinality of any $(n,\varepsilon)$ $i$-separated set in $\overline{W^i(x,\delta)}$.

    Let $R\subseteq\overline{W^i(x,\delta)}$ satisfying
    \[
    d(f^jy,f^jz)<\varepsilon, 0\leq j\leq n-1, \forall y,z\in S,
    \]
    we call $R$ an $(n,\varepsilon)$ $i$-spanning set of $\overline{W^i(x,\delta)}$. Let $S^i(f,\varepsilon,n,x,\delta)$ denote the smallest cardinality of any $(n,\varepsilon)$ $i$-spanning set in $\overline{W^i(x,\delta)}$.
  \end{definition}

  The following lemma gives us a relation between $N^i(f,\varepsilon,n,x,\delta)$ and \\ $S^i(f,\varepsilon,n,x,\delta)$.
  \begin{lemma}\label{SN}
    \[
    \underset{\varepsilon\to 0}{\lim}\underset{n\to\infty}{\limsup}\frac{1}{n}\log N^i(f,\varepsilon,n,x,\delta)=\underset{\varepsilon\to 0}{\lim}\underset{n\to\infty}{\limsup}\frac{1}{n}\log S^i(f,\varepsilon,n,x,\delta).
    \]
  \end{lemma}
  \begin{proof}
    cf.  the proof of Lemma 3.5 in \cite{Wu2015}.
  \end{proof}

  The estimation of $h^i_{\mu}(f)$ from above is based on the following lemma.
  \begin{lemma}\label{ergodic}
    Let $\Sigma\subset M$ with $\mu(\Sigma)=1$, and assume that $\lambda_i(x)$, $m_i(x)$ are constant when $x\in \Sigma$, then for any $\rho>0$, there exists $x\in \Sigma$ such that
    \[
      h^i_{\mu}(f)-\rho\leq\underset{\varepsilon\to 0}{\lim}\underset{n\to\infty}{\limsup}\frac{1}{n}\log S^i(f,\varepsilon,n,x,\delta).
    \]
  \end{lemma}
  \begin{proof}
    Let $\xi^i=\xi$ be any measurable partition subordinate to $i$th unstable foliation as in Lemma \ref{specialpartition}. Since $\mu$ is ergodic, then we can pick $x\in \Sigma$ with the following property: there exists a set $B\subset\xi(x)$ with $\mu^{\xi}_x(B)=1$, such that
     \[
     h^i_{\mu}(f,y,\xi)=h_{\mu}^i(f,\xi), \forall y\in B.
     \]

    In fact, $h^i_{\mu}(f,x,\xi)$ is $\mu$-a.e. constant, let $\Sigma_1$ be the set of $x\in M$ where $h^i_{\mu}(f,x,\xi)$ is constant. Let $\Sigma_2$ be the set of $x\in M$ such that $\mu^{\xi}_x(\xi(x))=1$, it is clear that $\mu(\Sigma_2)=1$. We can pick $x\in\Sigma\cap\Sigma_1\cap\Sigma_2$, and let $B=\xi(x)\cap\Sigma\cap\Sigma_1\cap\Sigma_2$. It is clear that $\mu^{\xi}_x(B)=1$ and B satisfies the property above.

    The property above implies that for any $\rho>0$ and $y\in B$, there exists $\varepsilon_0(y)$, such that if $0<\varepsilon<\varepsilon_0(y)$, then
    \begin{equation}\label{eqabove}
      \underset{n\to\infty}{\liminf}-\frac{1}{n}\log\mu^{\xi}_y(V^i(f,y,n,\varepsilon))\geq h_{\mu}^i(f,\xi)-\rho.
    \end{equation}

    Denote $B_\varepsilon:=\{y\in B|\varepsilon_0\geq\varepsilon\}$, then $B=\cup_{\varepsilon>0}B_\varepsilon$. So there exists $\varepsilon_1>0$ such that $\mu^{\xi}_x(B_\varepsilon)>1-\rho$ for any $\varepsilon<\varepsilon_1$. Fix such an $\varepsilon$. (\ref{eqabove}) implies that for any $y\in B_\varepsilon$ there exists $N=N(y)>0$ such that if $n\geq N$, then
    \[
    \mu^{\xi}_y(V^i(f,y,n,\varepsilon))\leq e^{-n(h_{\mu}(f,\xi)-\rho)}.
    \]

    Denote $B_\varepsilon^n:=\{y\in B_\varepsilon|N(y)\leq n\}$. Then $B=\cup_{n=1}^\infty B_\varepsilon^n$. So there exists $N$ large enough such that $\mu^{\xi}_x(B_\varepsilon^n)>\mu^{\xi}_x(B_\varepsilon)-\rho>1-2\rho$. Since $y\in\xi(x)$, $\mu^{\xi}_y=\mu^{\xi}_x$, for any $n\geq N$, one has
    \[
    \mu^{\xi}_x(V^i(f,y,n,\varepsilon))\leq e^{-n(h_{\mu}(f,\xi)-\rho)}.
    \]

    Now we take $\delta>0$ with $W^i(x,\delta)\supset\xi(x)$. Then there exists a set $R_n$ with cardinality no more than $S^i(f,\frac{\varepsilon}{2},n,x,\delta)$, such that
    \[
    \overline{W^i(x,\delta)}\cap B\subset\underset{z\in R_n}{\bigcup}V^i(f,z,n,\frac{\varepsilon}{2}),
    \]
    and $V^i(f,z,n,\frac{\varepsilon}{2})\cap B\neq \emptyset$. Choose an arbitrary point in $V^i(f,z,n,\frac{\varepsilon}{2})\cap B$ and denote it by $y(z)$. Then we have
    \begin{align*}
       1-2\rho &\leq \mu^{\xi}_x(\overline{W^i(x,\delta)}\cap B) \\
      {} &\leq \mu^{\xi}_x(\underset{z\in R_n}{\bigcup}V^i(f,z,n,\frac{\varepsilon}{2})) \\
      {} &\leq \sum_{z\in R_n}\mu^{\xi}_x(V^i(f,z,n,\frac{\varepsilon}{2})) \\
      {} &\leq \sum_{z\in R_n}\mu^{\xi}_x(V^i(f,y(z),n,\varepsilon)) \\
      {} &\leq S^i(f,\frac{\varepsilon}{2},n,x,\delta)e^{-n(h^i_{\mu}(f,\xi)-\rho)}.
    \end{align*}
And hence $S^i(f,\frac{\varepsilon}{2},n,x,\delta)\geq e^{n(h^i_{\mu}(f,\xi)-\rho)}$. Thus we have
    \[
    h^i_{\mu}(f)-\rho=h^i_{\mu}(f,\xi)-\rho\leq\underset{\varepsilon\to 0}{\lim}\underset{n\to\infty}{\limsup}\frac{1}{n}\log S^i(f,\varepsilon,n,x,\delta).
    \]
  \end{proof}

Now we begin the proof of Theorem \ref{Ruelle}. Let $S_n\subset\overline{W^i(x,\delta)}$ be an $(n,\varepsilon)$ $i$-separated set with the largest cardinality. When $n$ is large enough, we can pick $y_n\in S_n$ such that
  \begin{equation}\label{number}
    N^i(f,\varepsilon,n,x,\delta)\leq\frac{\tilde V_{x,\delta}}{\mathrm{Vol}(\exp^{-1}V^i(f,y_n,n,\frac{\varepsilon}{2}))},
  \end{equation}
  where $\exp_x$ is the exponential map at $x$, $\tilde V_{x,\delta}=\mathrm{Vol}(\exp^{-1}\overline{W^i(x,\delta)})$, and $\mathrm{Vol}(\cdot)$ denotes the volume function.

  Because of the compactness of $M$, we can choose $\delta>0$ small enough such that $\exp_x$ is a diffeomorphism and $d(\mathrm{exp_x}v,x)=\Vert v\Vert$ for $v\in T_xM$, $\Vert v\Vert<\delta$. In order to avoid a cumbersome computation, for every $x\in M$, we treat the tangent space $T_xM$ as it were $\mathbbm{R}^n$. We denote the Jacobian determinant of $\exp_x^{-1}f^{-1}\exp_{f(x)}|_{F(x)}$ at $y\in F(x)$ by $J_x(y)$. For any $\varepsilon>0$, we can choose $0<\delta(\varepsilon)<\frac{\delta}{2}$ such that $||J_x(y_1)|-|J_x(y_2)||<\varepsilon$, for any $x\in M$ and $y_1$, $y_2$$\in\pi_xB(0_x,\delta(\varepsilon))$, where $\pi_x\colon T_xM\to F(x)$ is the projection and $0_x$ is the null vector in $T_xM$. Let $\varepsilon_0\colon=\frac{1}{2}\inf\{|J_x(y)|:x\in M,y\in\pi_x\overline{B(0_x,\frac{\delta}{2})}\}$. Then for any $x\in M$, we have
  \begin{equation}\label{determinant}
    \frac{1}{2}<\frac{|J_x(y_1)|}{|J_x(y_2)|}<2,
  \end{equation}
  for any $y_1$, $y_2$$\in \pi_x B(0_x,\delta_0)$, where $\delta_0=\delta(\varepsilon)$.

  Fix $0<\varepsilon<\frac{\delta_0}{3}$, pick $\tilde{y}_n\in\overline{W^i(x,\delta)}$ such that $f^n(\tilde{y_n})\in B(f^n(y_n),\frac{\varepsilon}{2})\cap W^i(f^n(x))$ and let
  \[
  B^i_n=\exp_{f^n(\tilde{y}_n)}^{-1}B(f^n(y_n),\frac{\varepsilon}{2})\cap W^i(f^n(x)).
  \]
  Then we have
  \begin{align*}
  & \mathrm{Vol}(\exp_{\tilde{y}_n}^{-1}f^{-n}\exp_{f^n(\tilde{y}_n)}B^i_n) \\
  ={}& \mathrm{Vol}(\exp_{\tilde{y}_n}^{-1}f^{-1}\exp_{f(\tilde{y}_n)}\exp_{f(\tilde{y}_n)}^{-1}f^{-1}\exp_{f^2(\tilde{y}_n)}\cdots\exp_{f^{n-1}(\tilde{y}_n)}^{-1}f^{-1}\exp_{f^n(\tilde{y}_n)}B^i_n) \\
  ={}& \int_{B^i_n}|J_{\tilde{y}_n}(\tilde{f}_{n-1}(y))||J_{f\tilde{y}_n}(\tilde{f}_{n-2}(y))|\cdots|J_{f^{n-1}\tilde{y}_n}((y))|\mathrm{d}\lambda,
  \end{align*}
  where $\tilde{f}_j(y)=\exp_{f^{n-j}(\tilde{y}_n)}^{-1}f^{-j}\exp_{f^n(\tilde{y}_n)}$, $j=1,2,\cdots,n-1$ and $\lambda$ is the Lebesgue measure on $B^i_n$.

  Notice the definition of $B^i_n$, so we have
  \[
  \tilde{f}_{n-j}(y)\in\pi_xB(0_{f^{j-1}(\tilde{y}_n)},\frac{\varepsilon}{2}),
  \]
  for $j=1,2,\cdots,n-1$ and $y\in B^i_n$. Hence by (\ref{determinant}) we have
  \[
  |J_{f^{j-1}(\tilde{y}_n)}(\tilde{f}_{n-j}(y))|>\frac{1}{2}|J_{f^{j-1}(\tilde{y}_n)}(0_{f^{j-1}(\tilde{y}_n)})|.
  \]
  So we have
  \begin{align*}
     & \mathrm{Vol}(\exp_{\tilde{y}_n}^{-1}f^{-n}\exp_{f^n(\tilde{y}_n)}B^i_n) \\
    \geq{}& \frac{1}{2^n}\int_{B^i_n}\prod_{j=1}^{n-1}|J_{f^{j-1}(\tilde{y}_n)}(0_{f^{j-1}(\tilde{y}_n)})||J_{f^{n-1}(\tilde{y}_n)}(0_{\tilde{y}_n})|\mathrm{d}\lambda \\
    ={} & \frac{1}{2^n}\mathrm{Vol}(D_{f^n(\tilde{y}_n)}f^{-n}B^i_n).
  \end{align*}

  The last equality follows that $D\exp_x|_{0_x}$ is an identity.

  Let $R_{m,\varepsilon'}$ be the set of $x\in M$ such that for any $n\geq m$ and $v\in E_i(x)$, we have $\Vert D_xf^{-n}v \Vert\geq e^{n(-\lambda_i-\varepsilon')}$. By Oseledec's Theorem, we know that
  \[
  \lim_{m\to\infty}\mu(R_m,\varepsilon')=1.
  \]
  Without loss of generality, we assume that for any $y\in M$ and $a>0$, $\mu(B(y,a))>0$. In fact, let
  \[
  A=\{ y\in M\colon \exists a>0 \text{ such that }\mu(B(y,a))=0 \}.
  \]
  It is easy to check that $A$ is an $f$-invariant set and $\mu(A)=0$. So for $\varepsilon$, there exists $N>0$ such that for any $n\geq N$,
  \[
  B(y_n,\frac{\varepsilon}{2})\cap R_{n,\varepsilon'}\neq\emptyset.
  \]
  And for every $n\geq N$ we can choose an appropriate $x_n\in M$ such that
  \[
  B(y_n,\frac{\varepsilon}{2})\cap W^i(f^n(x_n))\cap R_{n,\varepsilon'}\neq\emptyset.
  \]
  So when $n$ is large enough such that
  \[
  \frac{1}{n}\log\frac{1}{C}+\frac{1}{n}\mathrm{dim}M\log\varepsilon<\varepsilon',
  \]
  we can pick $\tilde{y}_n$ from the set $B(y_n,\frac{\varepsilon}{2})\cap W^i(f^n(x_n))\cap R_{n,\varepsilon'}$. Hence when $n$ is large enough, we have
  \begin{align*}
  & -\frac{1}{n}\log\mathrm{Vol}(\exp_{\tilde{y}_n}^{-1}f^{-n}\exp_{f^n(\tilde{y}_n)}B^i_n) \\
  \leq{}& -\frac{1}{n}\log\frac{1}{2^n}\mathrm{Vol}(D_{f^n(\tilde{y}_n)}f^{-n}B^i_n) \\
  \leq{}& -\frac{1}{n}\log\frac{1}{2^n}C\prod_{j=1}^{u_i}\varepsilon e^{nm_j(-\lambda_j-\varepsilon')} \\
  \leq{}& \sum_{j=1}^{u_i}\lambda_jm_j+\log2+\mathrm{dim}M\varepsilon'+\frac{1}{n}\log\frac{1}{C}+\frac{1}{n}\mathrm{dim}M\log\varepsilon \\
  \leq{}& \sum_{j=1}^{u_i}\lambda_jm_j+\log2+(\mathrm{dim}M+1)\varepsilon'.
  \end{align*}
  where the constant $C$ only related to $f$.

  Since
  \begin{align*}
      & \underset{\varepsilon\to 0}{\lim}\underset{n\to\infty}{\limsup}\frac{1}{n}\log N^i(f,\varepsilon,n,x,\delta) \\
    \leq{} & \underset{\varepsilon\to 0}{\lim}\underset{n\to\infty}{\limsup}\frac{1}{n}(\log\widetilde{V}_{x, \delta}-\log\mathrm{Vol}(\exp_x^{-1}V^i(f,y_n,n,\frac{\varepsilon}{2}))),
  \end{align*}
  and notice that
  \[
  \underset{n\to\infty}{\limsup}\frac{1}{n}\log\widetilde{V}_{x, \delta}=0,
  \]
  so let $\varepsilon'\to0$, using Lemma \ref{SN} and Lemma \ref{ergodic}, we obtain
  \[
  h^i_{\mu}(f)-\rho\leq\sum_{j\leq u(i)}\lambda_jm_j+\log2.
  \]
Let $\rho\to 0$, we obtain
  \[
  h^i_{\mu}(f)\leq\sum_{j\leq u(i)}\lambda_jm_j+\log2.
  \]
For $N>0$, let $g=f^N$, then we have
  \[
  h^i_{\mu}(f)=\frac{1}{N}h^i_{\mu}(g)\leq\sum_{j\leq u(i)}\lambda_jm_j+\frac{1}{N}\log2.
  \]
  Let $N\to\infty$, we obtain
  \[
  h^i_{\mu}(f)\leq\sum_{j\leq u(i)}\lambda_jm_j.
  \]

  Now we have completed the proof of Theorem \ref{Ruelle}.
\section{Proof of Theorem \ref{mainresult}}
Now we start to prove Theorem \ref{mainresult}. By Theorem \ref{Ruelle}, we only need to complete the estimation of $h^i_{\mu}(f)$ from below. Firstly, we give the following lemma.
\begin{lemma}
  let $\lambda_x^i$ be the corresponding Riemannian measure on $W^i(x)$ and $\xi^i=\xi$ be a measurable partition as in Lemma \ref{specialpartition}. If for $\mu$-a.e. $x\in M$, $\mu^{\xi}_x\ll \lambda^i_x$, we have
  \begin{equation}\label{eq:lemma}
  h^i_{\mu}(f)\geq \int_{\Gamma_i}\underset{\varepsilon\to 0}{\lim}\underset{n\to\infty}{\liminf}\frac{1}{n}(-\log\lambda_x^i(V^i(f,x,n, \varepsilon)))\mathrm{d}\mu(x).
  \end{equation}
\end{lemma}
\begin{proof}
  Let $\alpha$ be a finite Borel partition of $M$ with $diam(\alpha)\leq\varepsilon$ and let $\alpha_n\colon=\alpha\vee f^{-1}\alpha\vee\dots\vee f^{-n}\alpha$.

  Set $\mathscr A$ be the $\sigma$-algebra generated by partitions $\alpha_n$, $n\geq 0$. Let $\tilde{\mu}_x^{\xi}$, $\tilde{\lambda}_x^i$ be two measures on $M$ satisfying
  \[
  \tilde{\mu}_x^{\xi}(M\setminus\xi(x))=0, \tilde{\mu}_x^{\xi}|_{\xi(x)}=\mu^{\xi}_x,
  \]
  and
  \[
  \tilde{\lambda}_x^i(M\setminus\xi(x))=0, \tilde{\lambda}_x^i|_{\xi(x)}=\mu^{\xi}_x.
  \]
It is easy to verify that $\tilde{\mu}_x^{\xi}\ll\tilde{\lambda}_x^i$.
  Let $k\colon M\to \mathbbm R$ be a $\tilde{\lambda}_x^i$-integrable function with respect to $\mathscr A$ such that
  \begin{equation}\label{eq:measure}
  \int_{A}kd\tilde{\lambda}_x^i=\tilde{\mu}_x^{\xi}(A), \forall A\in\mathscr A.
  \end{equation}
  Such a function exists because that $\tilde{\mu}_x^{\xi}\ll\tilde{\lambda}_x^i$. It follows from (\ref{eq:measure}) that
  \begin{equation}\label{measurelimit}
    \underset{n\to \infty}{\lim}\frac{\tilde{\mu}_x^{\xi}(\alpha_n(x))}{\tilde{\lambda}_x^i(\alpha_n(x))}=k(x), \text{for $\tilde{\lambda}_x^i$-a.e. $x\in M$}.
  \end{equation}
And hence we have
    \[
    -\frac{1}{n}\log\tilde{\mu}_x^{\xi}(\alpha_n(x))=-\frac{1}{n}\log\tilde{\lambda}_x^i(\alpha_n(x))-\frac{1}{n}\log \frac{\tilde{\mu}_x^{\xi}(\alpha_n(x))}{\tilde{\lambda}_x^i(\alpha_n(x))}.
    \]
Using (\ref{measurelimit}), we obtain
    \[
      \underset{n\to\infty}{\liminf}-\frac{1}{n}\log\tilde{\mu}_x^{\xi}(\alpha_n(x))=\underset{n\to\infty}{\liminf}\frac{1}{n}(-\log\tilde{\lambda}_x^i(\alpha_n(x))).
    \]
Note that $\{y \in M\colon d(f^k(y),f^k(x))<\varepsilon, 0\leq k\leq n-1\}=\colon V(f,x,n,\varepsilon) \supset \alpha_n(x)$, so we have
    \[
      \underset{n\to\infty}{\liminf}-\frac{1}{n}\log\tilde{\mu}_x^{\xi}(\alpha_n(x))=\underset{n\to\infty}{\liminf}\frac{1}{n}(-\log\tilde{\lambda}_x^i(V(f, x,n,\varepsilon))).
    \]
Observe that for $\varepsilon_0>0$ small enough, we can take $C>0$, such that for any $x\in M$,
    \[
    \frac{1}{C}d(y,z)\leq d^i(y,z)\leq Cd(y,z), \forall y,z \in \overline{W^i(x,\varepsilon_0)}.
    \]
Notice the relationship between $\tilde{\mu}_x^{\xi}$ and $\mu_x^{\xi}$, $\tilde{\lambda}_x^i$ and $\lambda_x^i$ respectively, if $\varepsilon\leq\varepsilon_0$ is small enough, we have
    \begin{equation}\label{newmeasurelimits}
    \underset{n\to\infty}{\liminf}-\frac{1}{n}\log\mu_x^{\xi}(\alpha_n(x)\cap\xi(x))\geq\underset{n\to\infty}{\liminf}\frac{1}{n}(-\log\lambda_x^i(V^i(f, x,n,C\varepsilon))).
    \end{equation}
Integrating both side of (\ref{newmeasurelimits}), we obtain
    \begin{align}\label{integration}
      & \int_{\Gamma_i}\underset{n\to\infty}{\liminf}-\frac{1}{n}\log\mu_x^{\xi}(\alpha_n(x)\cap\xi(x))\mathrm{d}\mu(x)\notag\\
      \geq{} & \int_{\Gamma_i}\underset{n\to\infty}{\liminf}\frac{1}{n}(-\log\lambda_x^i(V^i(f,x,n,C\varepsilon)))\mathrm{d}\mu(x).
    \end{align}\
    By Fatou's Lemma we obtain
    \begin{align}\label{eq:Fatou}
      \underset{n\to\infty}{\limsup}\frac{1}{n}H^i(\alpha_n|\xi) &= \underset{n\to\infty}{\limsup}\frac{1}{n}\int_{\Gamma_i}-\log\mu_x^{\xi}(\alpha_n(x)\cap\xi(x))\mathrm{d}\mu(x)\notag\\
      {} &\geq \underset{n\to\infty}{\liminf}\frac{1}{n}\int_{\Gamma_i}-\log\mu_x^{\xi}(\alpha_n(x)\cap\xi(x))\mathrm{d}\mu(x)\notag\\
      {} &\geq \int_{\Gamma_i}\underset{n\to\infty}{\liminf}-\frac{1}{n}\log\mu_x^{\xi}(\alpha_n(x)\cap\xi(x))\mathrm{d}\mu(x).
    \end{align}
    Hence by Proposition \ref{equivalence}, we have
    \begin{equation}\label{eq:Wu}
      h^i_{\mu}(f)\geq\underset{n\to\infty}{\limsup}\frac{1}{n}H^i(\alpha_n|\xi).
    \end{equation}
Combining (\ref{integration}), (\ref{eq:Fatou}), and (\ref{eq:Wu}), we obtain
    \begin{equation}\label{entropyint}
      h^i_{\mu}(f)\geq\int_{\Gamma_i}\underset{n\to\infty}{\liminf}\frac{1}{n}(-\log\lambda_x^i(V^i(f,x,n,C\varepsilon)))\mathrm{d}\mu(x).
    \end{equation}

Let$\{\varepsilon_k\}_{k\geq 1}$ be a sequence such that $\varepsilon_k>0$ and $\varepsilon\to 0$ as $k\to\infty$. Then by the monotone convergence theorem, we have
    \begin{align}\label{final}
      & \underset{\varepsilon\to0}{\lim}\int_{\Gamma_i}\underset{n\to\infty}{\liminf}\frac{1}{n}(-\log\lambda_x^i(V^i(f,x,n, C\varepsilon)))\mathrm{d}\mu(x)\notag\\
      ={} & \underset{k\to\infty}{\lim}\int_{\Gamma_i}\underset{n\to\infty}{\liminf}\frac{1}{n}(-\log\lambda_x^i(V^i(f,x,n, C\varepsilon_k)))\mathrm{d}\mu(x)\notag\\
      ={} & \int_{\Gamma_i}\underset{k\to\infty}{\lim}\underset{n\to\infty}{\liminf}\frac{1}{n}(-\log\lambda_x^i(V^i(f,x,n, C\varepsilon_k)))\mathrm{d}\mu(x)\notag\\
      ={} & \int_{\Gamma_i}\underset{\varepsilon\to0}{\lim}\underset{n\to\infty}{\liminf}\frac{1}{n}(-\log\lambda_x^i(V^i(f,x,n, C\varepsilon)))\mathrm{d}\mu(x).
    \end{align}

     Therefore (\ref{eq:lemma}) follows from (\ref{entropyint}) and (\ref{final}).
\end{proof}

Before going into the proof of Theorem \ref{mainresult}, we need a technical lemma from  \cite{SunTian2010}. In the statement of the lemma, we will use the following definition from  \cite{Mane1981}.
\begin{definition}
  Let $E$ be a normed space and $E=E_1\oplus E_2$ be a splitting. Define $\gamma(E_1,E_2)$ as the supremum of norms of the projections $\pi\colon E\to E_i$, $i=1, 2$,
  associated with the splitting. Moreover, we say that a subset $G\subset E$ is a $(E_1,E_2)$-graph if there exists an open set $U\subseteq E_2$ and a $C^1$ map $\psi\colon U\to E_1$ satisfying
  \[
  G=\{x+\psi(x)|x\in U\}.
  \]
The number $\sup\{\frac{\Vert\psi(x)-\psi(y)\Vert}{\Vert x-y\Vert}|x\neq y\in U\}$ is called the dispersion of $G$.
\end{definition}

The following lemma about graph transform on dominated bundles is a generalization of Lemma 3 in  \cite{Mane1981} by Sun and Tian \cite{SunTian2010}.
\begin{lemma}
  Given $\alpha>0$, $\beta>0$ and $c>0$, there exists $\tau>0$ with the following property. If $E$ is a finite-dimensional normed space and $E=E_1\oplus E_2$ a splitting with $\gamma(E_1,E_2)\leq\alpha$, and $\mathcal{F}$ is a $C^1$ embedding of a ball $B_{\delta}(0)\subset E$ into another Banach space $E'$ satisfying
\begin{enumerate}[label=(\arabic*)]
  \item $D_0\mathcal{F}$ is an isomorphism and $\gamma((D_0\mathcal{F})E_1, (D_0\mathcal{F})E_2)\leq\alpha$;
  \item $\Vert D_0\mathcal{F}-D_x\mathcal{F}\Vert\leq\tau$ for all $x\in B_{\delta}(0)$;
  \item $\frac{\Vert D_0\mathcal{F}|_{E_1}\Vert}{m(D_0\mathcal{F}|_{E_2})}\leq\frac{1}{2}$;
  \item $m(D_0\mathcal{F}|_{E_2})\geq\beta$;
\end{enumerate}
  then for every $(E_1,E_2)$-graph $G$ with dispersion $\leq c$ contained in the ball $B_{\delta(0)}$, its image $\mathcal{F}(G)$ is a $((D_0\mathcal{F})E_1,(D_0\mathcal{F})E_2)$-graph with dispersion$\leq c$.
  \end{lemma}

  The following lemma is also useful for the proof of Theorem \ref{mainresult}.
  \begin{lemma}\label{graph}
  Let $g\in \mathrm{Diff}^1(M)$ and $\Lambda$ be a $g$-invariant subset of $M$. If there is a $(1,i(x))$-dominated splitting on $\Lambda$: $T_{\Lambda}M=E\oplus F$, then for any $c>0$, there exists $\delta >0$ such that for every $x\in \Lambda$ and any $(E_x,F_x)$-graph $G$ with dispersion $\leq c$ contained in the Bowen ball $V(g,x,n,\delta)(n\geq 0)$, its image $g^n(G)$ is a $(D_xg^nE_x,D_xg^nF_x)$-graph with dispersion $\leq c$.
  \end{lemma}
  \begin{proof}
    cf.  the proof of Lemma 3.4 in \cite{SunTian2010}.
  \end{proof}

  Now we are ready to prove Theorem \ref{mainresult}.
  Fix any $\varepsilon>0$. Take $N_0$ so large that the set $\Gamma_{i,\varepsilon}=\{x\in\Gamma\colon D_i(x)\leq N_0\}$ has $\mu$-measure larger than $\mu(\Gamma_i)-\varepsilon$. Let $N=N_0!$ and $g=f^N$, then the splitting $T_{\Gamma_{i,\varepsilon}}M=E\oplus F$ satisfies $(1,i(x))$-dominated with respect to $g$:
  \[
  \frac{\Vert Dg|_{E(x)}\Vert}{m(Dg|_{F(x)})}\leq \prod^{\frac{N}{D_i(x)}-1}_{j=0}\frac{\Vert Df^{D_i(x)}|_{E(f^{jD_i(x)}(x))}\Vert}{m(Df^{D_i(x)}|_{F(f^{jD_i(x)}(x))})}\leq (\frac{1}{2})^{\frac{N}{D_i(x)}}\leq \frac{1}{2}, \forall x \in\Gamma_{i,\varepsilon}.
  \]

  Note that $\Gamma_{i,\varepsilon}$ is $f$-invariant and thus is $g$-invariant. In what follows, in order to avoid a cumbersome and conceptually unnecessary use of coordinate charts, we shall treat $M$ as if it were a Euclidean space, and let $\lambda$ be the Lebesgue measure on $M$. The reader will observe that all our arguments can be easily formalized by a completely straightforward use of local coordinates.

  Since dominated splitting can be extended on the closure of $\Gamma_{i,\varepsilon}$, and dominated splitting is always continuous(see \cite{BonattiDIazViana2005}), we can fix two constants $c>0$ and $a>0$ so small that if $x\in \Gamma_{i,\varepsilon}$, $y\in M$ and $d(x,y)<a$, then for every linear subspace $E\subseteq T_yM$ which is a $(E(x),F(x))$-graph with dispersion $<c$ we have
  \[
  \vert\log\vert\det(D_yg)|_E\vert-\log\vert\det(D_xg)|_{F(x)}\vert\vert\leq\varepsilon.
  \]
 Thus
  \begin{equation}\label{detinequality}
  \vert\det(D_yg)|_E\vert\geq\det\vert(D_xg)|_{F(x)}\vert\cdot e^{-\varepsilon}.
  \end{equation}

  By Lemma \ref{graph}, there exists $\delta\in (0,a)$ such that for every $x\in \Gamma_{i,\varepsilon}$ and any $(E_x,F_x)$-graph $G$ with dispersion $\leq c$ contained in the ball $V(g,x,n,\delta)(n\geq 0)$. Its image $g^n(G)$ is a $((D_xg^n)E_x,(D_xg^n)F_x)$-graph with dispersion $\leq c$.

  The estimation of $h^i_{\mu}(f)$ from below is based on the following fact.
  \begin{fact}
    For every $x\in \Gamma_{i,\varepsilon}$, we have
    \[
      \underset{n\to\infty}{\liminf}\frac{1}{n}(-\log\lambda_x^i(V^i(g,x,n,\delta)))\geq N\sum_{j\leq u(i)}m_j(x)\lambda_j(x)-\varepsilon.
    \]
  \end{fact}
  \begin{proof}
    Fix any $x\in \Gamma_{i,\varepsilon}$, there exists $B_x>0$ satisfying
    \begin{equation}\label{measurevs}
    \lambda_x^i(V^i(g,x,n,\delta))=B_x\lambda(V(g,x,n,\delta)).
    \end{equation}
It is clear that we can choose a positive constant $B_1$ such that $B_x\leq B_1$, for all $x\in M$. Fix $x\in\Gamma_{i,\varepsilon}$, now we consider the measure of $V(g,x,n,\delta)$, which we have that there is a constant $B>0$ satisfying
    \[
      \lambda(V(g,x,n,\delta))=B\int_{E(x)}\lambda[(y+F(x))\cap V(g,x,n,\delta)]\mathrm{d}\lambda(y),
    \]
    for all $n\geq 0$, where $\lambda$ also denotes the Lebesgue measure in the subspaces $E(x)$ and $y+F(x)$, $y\in E(x)$. Now we will show that
    \begin{equation}\label{eq:bridge}
    \underset{n\to\infty}{\liminf}\underset{y\in E(x)}{\inf}\frac{1}{n}[-\log\lambda(\Lambda_n(y))]\geq N\sum_{j\leq u(i)}m_j(x)\lambda_j(x)-\varepsilon,
    \end{equation}
    where
    \[
    \Lambda_n(y)=(y+F(x))\cap V(g,x,n,\delta).
    \]
If $\Lambda_n(y)$ is not empty, by Lemma \ref{graph}, we have $g^n(\Lambda_n(y))$ is a $(E(g^n(x)),F(g^n(x)))$-graph with dispersion $\leq c$.

Take $D>0$ such that $D>\mathrm{vol}(G)$ for every $(E(w),F(w))$-graph $G$ with dispersion $\leq c$ contained in $B_{\delta}(w)$,$w\in \Gamma_{i,\varepsilon}$. Observe that
    \[
    g^n(\Lambda_n(y))\subseteq g^nV(g,x,n,\delta)\subseteq B_{\delta}(g^n(x)),g^n(x)\in \Gamma_{i,\varepsilon}.
    \]
We have
    \[
    D>\mathrm{vol}(g^n(\Lambda_n(y)))=\int_{\Lambda_n(y)}\vert\det(D_zg^n)\vert|_{T_z\Lambda_n(y)}\vert \mathrm{d}\lambda(z).
    \]
Since
    \[
    g^j(\Lambda_n(y))\subseteq g^jV(g,x,n,\delta)\subseteq B_{\delta}(g^j(x))\subseteq B_a(g^j(x)), j=0,1,2, \dots, n,
    \]
    we have for any $z\in \Lambda_n(y)$,
    \[
    d(g^j(z),g^j(x))<a, j=0, 1, 2, \dots, n.
    \]
By (\ref{detinequality}), we have
    \begin{align*}
        & \vert\det(D_zg^n)|_{T_z\Lambda_n(y)}\vert \\
      ={} & \prod_{j=0}^{n-1}\vert\det(D_{g^j(z)}g)|_{T_{g^j(z)}g^j\Lambda_n(y)}\vert \\
      \geq{} & \prod_{j=0}^{n-1}[\vert\det(D_{g^j(x)}g)|_{F(g^j(x))}\vert\cdot e^{-\varepsilon}] \\
      ={} & \vert\det(D_xg^n)|_{F(x)}\vert\cdot e^{-n\varepsilon}.
    \end{align*}
Hence,
    \begin{align*}
      \frac{1}{n}\log D &\geq \frac{1}{n}\log \int_{\Lambda_n(y)}\vert\det(D_zg^n)\vert|_{T_z\Lambda_n(y)}\vert\mathrm{d}\lambda(z)\\
      {} &\geq \frac{1}{n}\log \int_{\Lambda_n(y)} \vert\det(D_xg^n)|_{F(x)}\vert\cdot e^{-n\varepsilon}\mathrm{d}\lambda(z)\\
      {} &= \frac{1}{n}\log [\lambda(\Lambda_n(y))\cdot\vert\det(D_xg^n)|_{F(x)}\vert\cdot e^{-n\varepsilon}]\\
      {} &= \frac{1}{n}\log \lambda(\Lambda_n(y))+\frac{1}{n}\log \vert\det(D_xg^n)|_{F(x)}\vert-\varepsilon.
    \end{align*}
It follows that
    \[
    \underset{n\to\infty}{\lim}-\frac{1}{n}\log \lambda(\Lambda_n(y))\geq\frac{1}{n}\log \vert\det(D_xg^n)|_{F(x)}\vert-\varepsilon.
    \]
Combining this inequality and the fact from Oseledec Theorem \cite{Oseledec1968}, we obtain
    \[
    \underset{n\to\infty}{\lim}-\frac{1}{n}\log\vert\det(D_xg^n)|_{F(x)}\vert=N\sum_{j\leq u(i)}m_j(x)\lambda_j(x),
    \]

    Now we have completed the proof of (\ref{eq:bridge}), then using (\ref{measurevs}), we obtain
    \begin{align*}
        & \underset{n\to\infty}{\liminf}\frac{1}{n}(-\log\lambda_x^i(V^i(g,x,n,\delta))) \\
      {}= & \underset{n\to\infty}{\liminf}\frac{1}{n}(-\log B_x\lambda(V(g,x,n,\delta))) \\
      {}\geq & \underset{n\to\infty}{\liminf}\frac{1}{n}(-\log B_1\lambda(V(g,x,n,\delta))) \\
      {}= & \underset{n\to\infty}{\liminf}\frac{1}{n}(-\log B_1B\int_{E(x)}\lambda[(y+F(x))\cap V(g,x,n,\delta)]\mathrm{d}\lambda(y)\\
      {}\geq & N\sum_{j\leq u(i)}m_j(x)\lambda_j(x)-\varepsilon,
    \end{align*}
    which completes the proof of the fact.
  \end{proof}

  Now we can complete the estimation of $h^i_{\mu}(f)$ from below.

  By (\ref{eq:lemma}), we have
  \[
  h^i_{\mu}(g)\geq\int_{\Gamma_i}\underset{\varepsilon\to 0}{\lim}\underset{n\to\infty}{\liminf}\frac{1}{n}(-\log\lambda_x^i(V^i(g,x,n, \varepsilon)))\mathrm{d}\mu(x).
  \]
  Using the {\bf{fact}}, we obtain
  \begin{align*}
    h^i_{\mu}(g) &\geq \int_{\Gamma_{i,\varepsilon}}N\sum_{j\leq u(i)}m_j(x)\lambda_j(x)-\varepsilon\mathrm{d}\mu(x) \\
    {} &= \int_{\Gamma_i}N\sum_{j\leq u(i)}m_j(x)\lambda_j(x)\mathrm{d}\mu(x)-\int_{\Gamma_i\setminus\Gamma_{i,\varepsilon}}N\sum_{j\leq u(i,x)}m_j(x)\lambda_j(x)\mathrm{d}\mu(x)\\
    {} &-\varepsilon\mu(\Gamma_{i,\varepsilon}) \\
    {} &\geq \int_{\Gamma_i}N\sum_{j\leq u(i)}m_j(x)\lambda_j(x)\mathrm{d}\mu(x)-NC\dim(M)\varepsilon-\varepsilon,
  \end{align*}
  where $C=\underset{x\in M}{\max}\log\Vert D_xf\Vert$.
  Hence,
  \[
  h^i_{\mu}(f)=\frac{1}{N}h^i_{\mu}(g)\geq\int_{M}N\sum_{j\leq u(i)}m_j(x)\lambda_j(x)\mathrm{d}\mu(x)-NC\dim(M)\varepsilon-\varepsilon.
  \]

  Since $\varepsilon$ is arbitrary, this completes the estimation of $h^i_{\mu}(f)$ from below.


\end{document}